\documentclass[preprint,10pt]{elsarticle}
\usepackage{amsfonts}

 \usepackage{graphics}
\usepackage{graphicx}
 \usepackage{epsfig}
\usepackage{amssymb}
\usepackage{amsthm}

\newtheorem{thm}{Theorem}

\newdefinition{rmk}{Remark}
\newproof{pf}{Proof}
\newdefinition{example}{Example}
\newdefinition{definition}{Definition}
\newdefinition{proposition}{Proposition}
\newdefinition{corollary}{Corollary}

\journal{Arkiv}

\begin{document}

\begin{frontmatter}



\title{Exclusion sets in eigenvalue inclusion sets for tensors}



\author[rvt]{Chaoqian Li}
\author[rvt]{Suhua Li}
\author[rvt1]{Qingbing Liu}
\author[rvt]{Yaotang Li\corref{cor1}}
\ead{liyaotang@ynu.edu.cn}

\cortext[cor1]{Corresponding author.}

\address[rvt]{School of Mathematics and Statistics, Yunnan
University, Kunming, Yunnan,  P. R. China 650091}

\address[rvt1]{Department of
Mathematics, Zhejiang Wanli University, Ningbo, P.R. China }

\begin{abstract}
By excluding some sets, which don't include any eigenvalue of a tensor, from some existing eigenvalue inclusion sets,
two new sets are given to locate all eigenvalues of a tensor.
And it is shown that these two sets are contained in the Ger\v{s}gorin eigenvalue inclusion set of tensors provide by Qi (Journal of
Symbolic Computation  2005; 40:1302-1324)
and the Brauer-type eigenvalue inclusion set provide by Li et al. (Numer. Linear Algebra Appl. 2014; 21:39-50) respectively.
Two sufficient conditions such that the determinant of a tensor is not zero are also provided.
\end{abstract}

\end{frontmatter}


\section{Introduction}
 We call $\mathcal
{A}=(a_{i_1\cdots i_m})$ a complex (real) tensor of order $m$
dimension $n$, denoted by $ \mathcal
{A}\in \mathbb{C}^{[m,n]}$ $(\mathcal{A}\in \mathbb{R}^{[m,n]}$), if
\[a_{i_1\cdots i_m}\in \mathbb{C}~ (\mathbb{R}),\] where $i_j=1,\ldots,n$ for
$j=1,\ldots, m$. Obviously, a vector is a tensor of order $1$ and a
matrix is a tensor of order $2$.  A real tensor $\mathcal {A}=(a_{i_1\cdots i_m })$ is
called symmetric \cite{Qi3} if
\[a_{i_1\cdots i_m }= a_{\pi(i_1\cdots i_m )},\forall \pi\in
\Pi_m,\]where $\Pi_m$ is the permutation group of $m$ indices. Furthermore,
a complex number $\lambda$ is called an eigenvalue of $\mathcal {A}=(a_{i_1\cdots i_m}) \in \mathbb{C}^{[m,n]}$ and a nonzero complex vector $x$ an
eigenvector of $\mathcal {A}$ associated with $\lambda$ if $\lambda$ and $x$ satisfy
\begin{equation}\label{Maineq}
\mathcal {A}x^{m-1}=\lambda x^{[m-1]}, \end{equation}
where  \[(\mathcal
{A}x^{m-1})_i=\sum\limits_{i_2,\ldots,i_m\in N} a_{ii_2\cdots
i_m}x_{i_2}\cdots x_{i_m},~x^{[m-1]}=(x_1^{m-1},x_2^{m-1},\ldots,x_n^{m-1})^T\] and $N=\{1,2,$$\ldots,n\}$.
This definition was introduced by Qi in \cite{Qi3} where he assumed that
$\mathcal {A}$ is an order $m$ dimension $n$ symmetric tensor
and $m$ is even. Independently, in \cite{Li}, Lim gave such a
definition but restricted $x$ to be a real vector and $\lambda$ to
be a real number. In this case, we call $\lambda$ an H-eigenvalue of
$\mathcal {A}$ and $x$ an H-eigenvector of $\mathcal {A}$ associated
with $\lambda$ \cite{Qi6,Qi3}. Note that there are other
definitions of eigenvalue and eigenvectors, such as, $D$-eigenvalue
and $Z$-eigenvalue; see \cite{Ca1,Ko1,Qi4,Qi1,Qi5,Qi2,Wa}.

One of the important problems on eigenvalues of a tensor is to locate all its eigenvalues, i.e.,
to give a set including all its eigenvalues in the complex plane. The first work owes to Liqun Qi.
He in \cite{Qi3} gave an eigenvalue inclusion set for real symmetric tensors, which is a generalization of the well-known Ger\v{s}gorin set
of matrices \cite{Ger,Varga,Varga1}. Subsequently, Li et al. provided some Brauer-type eigenvalue inclusion sets for general tensors \cite{Liet0,Liet1,Liet2}, and shown thatthe Brauer-type eigenvalue inclusion sets capture all eigenvalues of a tensor precisely than the set given by Qi.
Very recently, Bu et al. extend  the Brualdi set of matrices to  higher order tensors. In addition, another eigenvalue inclusion sets were also given, for details, see \cite{Bu,Bu1,Hu1,Liet3}.

When constructing these existing eigenvalue inclusion sets, one didn't consider the problem that
whether or not there is some proper subset of these sets in which each eigenvalue of a tensor is not included.
In this paper, by using (\ref{Maineq}) and the eigenvector corresponding an eigenvalue of a tensor,
we give some such sets, and exclude them respectively from the Ger\v{s}gorin set of tensors in \cite{Qi3} and the Brauer-type eigenvalue inclusion set in \cite{Liet0}
to give two new sets including all eigenvalues of a tensor.
As applications, two sufficient conditions such that the determinant of a tensor is not zero are also provided.

\section {Exclusion sets in the Ger\v{s}gorin set for tensors}
In \cite{Qi3}, Qi extended the well-known Ger\v{s}gorin's eigenvalue
inclusion theorem \cite{Ger,Varga,Varga1} of matrices to real
symmetric tensors.  This result can be easily generalized to
general tensors \cite{Ya} (see Theorem \ref{th2.1}).

\begin{thm}
\label{th2.1} Let $\mathcal {A}=(a_{i_1\cdots i_m})\in \mathbb{C}^{[m,n]}$. Then
\[
\sigma(\mathcal {A})\subseteq \Gamma(\mathcal
{A})=\bigcup\limits_{i\in N} \Gamma_i(\mathcal {A}),
\]
where $\sigma (\mathcal {A})$ is the set of all the eigenvalues of
$\mathcal{A}$,\[\Gamma_i(\mathcal {A})=\left\{z\in
\mathbb{C}:|z-a_{i\cdots i}|\leq r_i(\mathcal
{A})\right\},~r_i(\mathcal {A})= \sum\limits_{i_2,\ldots,i_m\in
N,\atop \delta_{ii_2\ldots i_m}=0} |a_{ii_2\cdots i_m}|,\]
and \[\delta_{i_1\cdots i_m}=\left\{\begin{array}{cc}
   1,   &if~ i_1=\cdots =i_m,  \\
   0,  &otherwise.
\end{array}
\right.\]
\end{thm}

$\Gamma_i(\mathcal {A})$ is a disk in the complex plane with $a_{i\cdots i}$ as its center and $r_i(\mathcal
{A})$ as their radii. Obviously,  $\Gamma(\mathcal{A})$ consists of $n$ disks.
The proof of Theorem \ref{th2.1} relies on (\ref{Maineq}), and is listed as follows,
which is useful for getting some exclusion sets.

\textbf{The proof of Theorem \ref{th2.1}} Suppose that $\lambda \in \sigma (\mathcal {A})$ with a corresponding
eigenvector $x=(x_1,x_2,\ldots,x_n)^T$. Let
\[|x_t|=\max\limits_{i\in N} |x_i|.\]
Consider the $t$th equation of (\ref{Maineq}). We have
\[(\lambda-a_{t\cdots t})
x_t^{m-1}=\sum\limits_{\delta_{ti_2\ldots i_m}=0} a_{ti_2\cdots i_m}x_{i_2}\cdots x_{i_m}. \]
Taking absolute values on both sides and using the triangle inequality yields
\[ |\lambda-a_{t\cdots t}|
|x_t|^{m-1}\leq\sum\limits_{\delta_{ti_2\ldots i_m}=0} |a_{ti_2\cdots i_m}||x_{i_2}|\cdots |x_{i_m}|\leq \sum\limits_{\delta_{ti_2\ldots i_m}=0} |a_{ti_2\cdots i_m}||x_{t}|^{m-1}=r_i(\mathcal
{A})|x_{t}|^{m-1}.\]
Hence,
\[ |\lambda-a_{t\cdots t}|\leq r_t(\mathcal
{A}),\]
that is, \begin{equation}\label{eqn2.1}\lambda \in \Gamma_t(\mathcal {A}).\end{equation} We do not know which $t$ each eigenvalue corresponds to, hence we have
$\lambda \in \bigcup\limits_{i\in N} \Gamma_i(\mathcal {A})$, consequently, $\sigma(\mathcal {A})\subseteq \Gamma(\mathcal
{A})$. $\Box$

It is easy to see that we only use the largest modulus $|x_t|$ of the eigenvector $x$ and  the $t$-th equation of (\ref{Maineq}) in the proof of Theorem \ref{th2.1}.
However, the other components of the eigenvector $x$ and the other equations of (\ref{Maineq}) are not considered, which may result in losing some informations on
$\Gamma(\mathcal {A})$. Next, by considering the other components $x_j$ of the eigenvector $x$ with $j\neq t$, we give an improvement of  $\Gamma(\mathcal {A})$.

\begin{thm}
\label{th2.2} Let $\mathcal {A}=(a_{i_1\cdots i_m})\in \mathbb{C}^{[m,n]}$. Then
\[
\sigma(\mathcal {A})\subseteq \Omega(\mathcal
{A})=\bigcup\limits_{i\in N} \Omega_i(\mathcal {A}),
\]
where  $\Omega_i(\mathcal {A})= \Gamma_i(\mathcal {A})\backslash \Delta_i(\mathcal {A})$,
\[ \Delta_i(\mathcal {A})= \bigcup\limits_{j\neq i} \Delta_{ij}(\mathcal {A}) \]
and
\[\Delta_{ij}(\mathcal {A})= \left\{z\in
\mathbb{C}:|z-a_{j\cdots j}| < 2|a_{ji\cdots i}|- r_j(\mathcal
{A})\right\}.\]
Furthermore, $\Omega(\mathcal
{A}) \subseteq  \Gamma(\mathcal {A})$.
\end{thm}

\begin{proof} Suppose that $\lambda \in \sigma (\mathcal {A})$ with a corresponding
eigenvector $x=(x_1,x_2,\ldots,x_n)^T$. According to the proof of Theorem \ref{th2.1}, (\ref{eqn2.1}) holds. Furthermore, for any $j\in N$ and $j\neq t$, we have
\[|x_j|\leq |x_t|,\]
and
\[(\lambda-a_{j\cdots j})
x_j^{m-1}=\sum\limits_{\delta_{ji_2\ldots i_m}=0,\atop \delta_{ti_2\ldots i_m}=0} a_{ji_2\cdots i_m}x_{i_2}\cdots x_{i_m}+a_{jt\cdots t}x_{t}^{m-1}. \]
Hence, \[a_{jt\cdots t}x_{t}^{m-1}=(\lambda-a_{j\cdots j})
x_j^{m-1}- \sum\limits_{\delta_{ji_2\ldots i_m}=0,\atop \delta_{ti_2\ldots i_m}=0} a_{ji_2\cdots i_m}x_{i_2}\cdots x_{i_m} \]
and
\begin{eqnarray*} |a_{jt\cdots t}||x_{t}|^{m-1}&\leq&|\lambda-a_{j\cdots j}|
|x_j|^{m-1}+\sum\limits_{\delta_{ji_2\ldots i_m}=0,\atop \delta_{ti_2\ldots i_m}=0} |a_{ji_2\cdots i_m}||x_{i_2}|\cdots |x_{i_m} | \\
&\leq & |\lambda-a_{j\cdots j}|
|x_t|^{m-1}+\sum\limits_{\delta_{ji_2\ldots i_m}=0,\atop \delta_{ti_2\ldots i_m}=0} |a_{ji_2\cdots i_m}||x_{t}|^{m-1},
\end{eqnarray*}
which implies
\[ |a_{jt\cdots t}| \leq |\lambda-a_{j\cdots j}| + \sum\limits_{\delta_{ji_2\ldots i_m}=0,\atop \delta_{ti_2\ldots i_m}=0} |a_{ji_2\cdots i_m}|\]
and
\[|\lambda-a_{j\cdots j}| \geq |a_{jt\cdots t}|-\sum\limits_{\delta_{ji_2\ldots i_m}=0,\atop \delta_{ti_2\ldots i_m}=0} |a_{ji_2\cdots i_m}|= 2|a_{jt\cdots t}|-r_j(\mathcal
{A}),\]
i.e.,
\begin{equation}\label{eqn2.2}  \lambda \notin \Delta_{tj}(\mathcal {A}).\end{equation}
Note that (\ref{eqn2.2}) holds for any $j\neq t$. Then
\begin{equation}\label{eqn2.3}\lambda \notin \bigcup\limits_{j\neq t} \Delta_{tj}(\mathcal {A})= \Delta_{t}(\mathcal {A}).\end{equation}
Combining (\ref{eqn2.1}) and  (\ref{eqn2.3}) gives
\[ \lambda \in \Gamma_t(\mathcal {A}) \backslash \Delta_{t}(\mathcal {A})=\Omega_t(\mathcal {A}),\]
consequently,
$ \lambda \in \bigcup\limits_{i\in N} \Omega_i(\mathcal {A})=\Omega(\mathcal
{A}) $ and $ \sigma(\mathcal {A})\subseteq \Omega(\mathcal
{A})$.

Moreover, from \[ \Omega_i(\mathcal {A})= \Gamma_i(\mathcal {A})\backslash \Delta_i(\mathcal {A}) \subseteq  \Gamma_i(\mathcal {A})\] we can easily obtain
$\Omega(\mathcal{A}) \subseteq  \Gamma(\mathcal {A})$. The proof is completed. \end{proof}

\begin{rmk} \label{rmk2.1} Note that
$|a_{ji\cdots i}|\leq  r_j(\mathcal
{A}) $ and $2|a_{ji\cdots i}|- r_j(\mathcal
{A})\leq  r_j(\mathcal{A}) $. Hence, \[\Delta_{ij}(\mathcal {A}) \subseteq \Gamma_j(\mathcal {A}), j\neq i, j\in N\] and
\[\Delta_i(\mathcal {A})= \bigcup\limits_{j\neq i} \Delta_{ij}(\mathcal {A}) \subseteq \Gamma(\mathcal {A}).\]
On the other hand, it is shown by Theorem \ref{th2.2} that $\Delta_i(\mathcal {A})$ dose not include any eigenvalues of a tensor $A$, and $\Omega(\mathcal {A})$ is obtained by excluding
some proper subsets $\Delta_i(\mathcal {A})$ from the Ger\v{s}gorin set $ \Gamma(\mathcal {A})$. And hence $\Delta_i(\mathcal {A})$ is a so-called exclusion set for the Ger\v{s}gorin set $  \Gamma(\mathcal {A})$.

 Consider the tensor $\mathcal {A}=(a_{ijk})\in \mathbb{C}^{[3,4]}$,where
\[ a_{111}=12, a_{222}=14, a_{333}=8+\textbf{i}, a_{444}=11, \]
\[a_{122}=4+\textbf{i},a_{144}=15-\textbf{i},a_{233}=5-\textbf{i},a_{211}=-2-\textbf{i},\]
\[a_{322}=6, a_{344}=4,~a_{411}=16,a_{422}=2,\]
and other $a_{ijk}=0$. The sets $\Omega_1(\mathcal {A})$, $\Omega_2(\mathcal {A})$, $\Omega_3(\mathcal {A})$ and $\Omega_4(\mathcal {A})$ are drawn in Figure 1.
And their union $ \Omega(\mathcal {A})$ are drawn in Figure 2.
The exact eigenvalues of $ \mathcal {A}$ are plotted with asterisks, which are computed by the MATLAB code \textit{solve}. It is not difficult to see that each $\Delta_i(\mathcal {A})$, $i=1,2,3,4$ does not include any eigenvalues of $\mathcal {A}$, but
$\Omega(\mathcal {A})$ does, and that $\Omega(\mathcal {A})$ is a proper subset of $\Gamma(\mathcal {A}) $, i.e., $\Omega(\mathcal {A}) \subset  \Gamma(\mathcal {A}) $.


\end{rmk}

The determinant of a tensor $\mathcal {A}\in \mathbb{C}^{[m,n]}$, denoted by $det (\mathcal {A})$, is the resultant of the ordered system
of homogeneous equations $\mathcal {A} x^{m-1} =0$ \cite{Hu}, and is be closely related to the eigenvalue inclusion set of a tensor. Next, Based on Theorem \ref{th2.2} and the fact that
$ det (\mathcal {A})= 0$ if and only if $0\in \sigma(\mathcal {A})$  for a tensor $\mathcal {A}$ \cite{Hu}, we can easily obtain the following condition such that $det (\mathcal {A}) \neq 0 $.

\begin{corollary} \label{cor0} Let $\mathcal {A} =(a_{i_1i_2\cdots i_n})\in \mathbb{C}^{[m,n]}$. If for each $i\in N$,
either
\[ |a_{i\cdots i}|> r_i(\mathcal{A})\]
or
\[|a_{j\cdots j}| < 2|a_{ji\cdots i}|- r_j(\mathcal
{A})~for ~some j\neq i,\]
then $det (\mathcal {A}) \neq 0 $.
\end{corollary}

A matrix is a tensor of order $2$. Hence, when $m=2$, Theorem \ref{th2.2} reduces to the following result.

\begin{corollary} \label{cor1} Let $A=(a_{ij})$ be a complex matrix. Then
\[
\sigma(A)\subseteq \Omega(A)=\bigcup\limits_{i\in N} \Omega_i(A),
\]
where  $\Omega_i(A)= \Gamma_i(A)\backslash \Delta_i(A)$,
\[ \Delta_i(A)= \bigcup\limits_{j\neq i} \Delta_{ij}(A) \]
and
\[\Delta_{ij}(A)= \left\{z\in
\mathbb{C}:|z-a_{jj}| < 2|a_{ji}|- r_j(A)\right\}.\]
Furthermore, $\Omega(A) \subseteq  \Gamma(A)$.
\end{corollary}

Remak here that the set $\Omega(A)$ in Corollary \ref{cor1} is a correction of the eigenvalue inclusion set
\[\bigcup\limits_{i\in N} \left( \Gamma_i(A)\backslash  \left(\bigcup\limits_{j\neq i} \Delta_{ij}'(A)\right)\right) \] for matrices in \cite{Me}, where
\[\Delta_{ij}'(A)= \left\{z\in
\mathbb{C}:|z-a_{jj}| \geq  2|a_{ji}|- r_j(A)\right\}.\]

\section {Exclusion sets for the Brauer-type set for tensors}

Another well-known eigenvalue inclusion set for matrices are provided by Brauer in \cite{brauer}.
In \cite{Liet0} Li et al. gave an example to show that this set
cannot be extended to higher order tensors, and gave a Brauer-type set to locate all eigenvalues of a tensor as follows.

\begin{thm}
\label{th3.1} Let $\mathcal {A}=(a_{i_1\cdots i_m}) \in \mathbb{C}^{[m,n]}$ with $n\geq 2$. Then
\[
\sigma(\mathcal {A})\subseteq \mathcal{K}(\mathcal {A})=
\bigcup\limits_{i,j\in N,\atop j\neq i} \mathcal {K}_{ij}(\mathcal
{A}),
\]
where
\[\mathcal {K}_{ij}(\mathcal {A})=\left\{z\in
\mathbb{C}:\left(|z-a_{i\cdots i}|-r_i^j(\mathcal
{A})\right)|z-a_{j\cdots j}|\leq |a_{ij\cdots j}|r_j(\mathcal
{A})\right\}\] and
\[r_i^j(\mathcal
{A})=\sum\limits_{\delta_{ii_2\ldots i_m}=0,\atop \delta_{ji_2\ldots
i_m}=0} |a_{ii_2\cdots i_m}|=r_i(\mathcal {A})-|a_{ij\cdots j}|.\]
\end{thm}

Next we try to find some proper subsets of $\mathcal{K}(\mathcal {A})$  in which there is not any eigenvalue of a tensor $\mathcal {A}$.

\begin{thm}
\label{nth3.1}  Let $\mathcal {A}=(a_{i_1\cdots i_m}) \in \mathbb{C}^{[m,n]}$ with $n\geq 2$. Then
\[
\sigma(\mathcal {A})\subseteq \Theta (\mathcal
{A}) =\bigcup\limits_{i,j\in N,\atop j\neq i} \Theta_{ij} (\mathcal
{A}),\]
where $ \Theta_{ij} (\mathcal {A})=\mathcal {K}_{ij}(\mathcal
{A})\backslash \Lambda_{i}(\mathcal {A}) $, $\Lambda_{i}(\mathcal {A})=\bigcup\limits_{p\neq i} \Lambda_{ip}(\mathcal {A})$
and
\[\Lambda_{ip}(\mathcal {A})=\left\{z\in
\mathbb{C}:(|z-a_{i\cdots i}|+r_i^p(\mathcal {A}))|z-a_{p\cdots p}|< |a_{ip\cdots p}|(2|a_{pi\cdots i}|- r_p(\mathcal {A})) \right\}.  \]
Furthermore, $ \Theta (\mathcal{A}) \subseteq \mathcal {K}(\mathcal{A})$.
\end{thm}

\begin{proof} For any $\lambda\in \sigma(\mathcal {A})$, let $ x=(x_1,x_2,\ldots,x_n)^T\in
\mathbb{C}^n\backslash \{0\}$ be an associated eigenvector, i.e.,
\[
\mathcal {A}x^{m-1}=\lambda x^{[m-1]}. \] Let
\[|x_t|\geq |x_s|\geq \max\{|x_k|:k\in N,k\neq s, k\neq t\}\] (where
the last term above is defined to be zero if $n=2$). Obviously,
$|x_t|>0$. From (\ref{Maineq}), we have
\[(\lambda-a_{t\cdots t})
x_t^{m-1}=\sum\limits_{\delta_{ti_2\ldots i_m}=0,\atop
\delta_{si_2\ldots i_m} =0} a_{ti_2\cdots i_m}x_{i_2}\cdots x_{i_m}+
a_{ts\cdots s}x_s^{m-1}. \] Taking modulus in the above
equation and using the triangle inequality gives
\begin{eqnarray*}
|\lambda-a_{t\cdots t}||x_t|^{m-1}&\leq &
\sum\limits_{\delta_{ti_2\ldots i_m}=0,\atop
\delta_{si_2\ldots i_m}=0} |a_{ti_2\cdots i_m}||x_{i_2}|\cdots |x_{i_m}|+|a_{ts\cdots s}||x_s|^{m-1}\nonumber\\
&\leq& \sum\limits_{\delta_{ti_2\ldots i_m}=0,\atop
\delta_{si_2\ldots
i_m}=0} |a_{ti_2\cdots i_m}||x_t|^{m-1}+|a_{ts\cdots s}||x_s|^{m-1}\\
&=& r_t^s(\mathcal {A})|x_t|^{m-1}+|a_{ts\cdots
s}||x_s|^{m-1},\nonumber
\end{eqnarray*}
equivalently, \begin{equation}\label{eq 3.2} \left(
|\lambda-a_{t\cdots t}|-r_t^s(\mathcal {A})\right)|x_t|^{m-1}\leq
|a_{ts\cdots s}||x_s|^{m-1}.
\end{equation}
If $|x_s|=0$, then $|\lambda-a_{t\cdots t}|-r_t^s(\mathcal {A})\leq
0$ as $|x_t|>0$, and it is obvious that $\lambda \in \mathcal
{K}_{t,s}(\mathcal {A})\subseteq \mathcal {K}(\mathcal {A})$.
Otherwise, $|x_s|>0$. Moreover, from (\ref{Maineq}), we
similarly get
\begin{equation}\label{eq 3.3} |\lambda-a_{s\cdots s}||x_s|^{m-1}\leq r_s(\mathcal {A})|x_t|^{m-1}.
\end{equation}
Multiplying Inequality (\ref{eq 3.3}) with Inequality (\ref{eq
3.2}), we have
\[\left(|\lambda-a_{t\cdots t}|-r_t^s(\mathcal {A})\right)|\lambda-a_{s\cdots s}||x_t|^{m-1}|x_s|^{m-1}\leq
|a_{ts\cdots s}|r_s(\mathcal {A})|x_t|^{m-1}|x_s|^{m-1}.\] Note that
$|x_t|^{m-1}|x_s|^{m-1}>0$. Then
\[\left(|\lambda-a_{t\cdots t}|-r_t^s(\mathcal {A})\right)|\lambda-a_{s\cdots s}|\leq
|a_{ts\cdots s}|r_s(\mathcal {A}),\] which implies
\begin{equation}\label{eqn2.4} \lambda\in \mathcal {K}_{ts}(\mathcal {A}).\end{equation}

By the $p$-th equation of (\ref{Maineq}) for each  $p\neq t$, we have
\[(\lambda-a_{p\cdots p})
x_p^{m-1}-\sum\limits_{\delta_{pi_2\ldots i_m}=0,\atop
\delta_{ti_2\ldots i_m} =0} a_{pi_2\cdots i_m}x_{i_2}\cdots x_{i_m}=
a_{pt\cdots t}x_t^{m-1} \]
and
\[|a_{pt\cdots t}||x_t|^{m-1}\leq |\lambda-a_{p\cdots p}| |x_p|^{m-1} +
\sum\limits_{\delta_{pi_2\ldots i_m}=0,\atop
\delta_{ti_2\ldots i_m} =0} |a_{pi_2\cdots i_m}||x_{t}|^{m-1},\]
equivalently,
\begin{equation} \label{eqn2.5} (2|a_{pt\cdots t}|- r_p(\mathcal {A}))|x_t|^{m-1}\leq |\lambda-a_{p\cdots p}| |x_p|^{m-1}. \end{equation}
Similarly, by the $t$th equation of (\ref{Maineq}), we have
\[(\lambda-a_{t\cdots t})
x_t^{m-1}-\sum\limits_{\delta_{ti_2\ldots i_m}=0,\atop
\delta_{pi_2\ldots i_m} =0} a_{ti_2\cdots i_m}x_{i_2}\cdots x_{i_m}=
a_{tp\cdots p}x_p^{m-1} \]
and
\[|a_{tp\cdots p}||x_p|^{m-1} \leq |\lambda-a_{t\cdots t}||x_t|^{m-1}+
\sum\limits_{\delta_{ti_2\ldots i_m}=0,\atop
\delta_{pi_2\ldots i_m} =0} |a_{ti_2\cdots i_m}||x_{t}|^{m-1},\]
equivalently,
\begin{equation} \label{eqn2.6} |a_{tp\cdots p}||x_p|^{m-1} \leq (|\lambda-a_{t\cdots t}|+r_t^p(\mathcal {A}))|x_t|^{m-1}. \end{equation}
If $|x_p|>0$, then multiplying Inequality (\ref{eqn2.5}) with Inequality (\ref{eqn2.6}) gives
\[(2|a_{pt\cdots t}|- r_p(\mathcal {A}))|a_{tp\cdots p}||x_p|^{m-1}|x_t|^{m-1} \leq
|\lambda-a_{p\cdots p}|(|\lambda-a_{t\cdots t}|+r_t^p(\mathcal {A}))|x_t|^{m-1} |x_p|^{m-1}\]
and
\begin{equation} \label{eqn2.7}
(2|a_{pt\cdots t}|- r_p(\mathcal {A}))|a_{tp\cdots p}| \leq
|\lambda-a_{p\cdots p}|(|\lambda-a_{t\cdots t}|+r_t^p(\mathcal {A})),
\end{equation}
i.e.,
\begin{equation} \label{eqn2.8}
\lambda \notin \Lambda_{tp}(\mathcal {A}).
\end{equation}
If $|x_p|=0$, then $ 2|a_{pt\cdots t}|- r_p(\mathcal {A})\leq 0$ holds from (\ref{eqn2.5}),  and (\ref{eqn2.7}) also holds, consequently, (\ref{eqn2.8}) holds.
Note that  (\ref{eqn2.8}) holds for any $p\neq t$. Hence,
 \begin{equation} \label{eqn2.9}
\lambda \notin  \bigcup \limits_{p\neq t} \Lambda_{tp}(\mathcal {A}).
\end{equation}
By (\ref{eqn2.4}) and (\ref{eqn2.9}) we have
\[\lambda \in
{K}_{ts}(\mathcal {A}) \backslash \left(\bigcup \limits_{p\neq t}\Lambda_{tp}(\mathcal {A})\right),\] this implies
\[\lambda \in \left(
\bigcup\limits_{i,j\in N,\atop j\neq i} \mathcal {K}_{ij}(\mathcal
{A})\backslash \left(\bigcup \limits_{p\neq i}\Lambda_{ip}(\mathcal {A})\right) \right)
 =\bigcup\limits_{i,j\in N,\atop j\neq i} \mathcal {K}_{ij}(\mathcal
{A})\backslash \Lambda_{i}(\mathcal {A}) = \bigcup\limits_{i,j\in N,\atop j\neq i} \Theta_{ij} (\mathcal
{A}) =\Theta (\mathcal
{A}).\]

Furthermore, from $ \Theta_{ij} (\mathcal {A})=\mathcal {K}_{ij}(\mathcal
{A})\backslash \Lambda_{i}(\mathcal {A}) \subseteq \mathcal {K}_{ij}(\mathcal
{A}) $, we have $ \Theta (\mathcal
{A}) \subseteq \mathcal {K}(\mathcal
{A})$. The conclusion follows.
\end{proof}

From Remak \ref{rmk2.1}, we have that for any $p\neq i$,
$2|a_{pi\cdots i}|- r_p(\mathcal {A})\leq r_p(\mathcal {A}) $
and then
\[\Lambda_{ip}(\mathcal {A}) \subseteq \mathcal {K}_{ip}(\mathcal
{A}).\]
However, $\lambda \notin \Lambda_{ip}(\mathcal {A})$ for each $\lambda \in \sigma(\mathcal {A})$.
Hence,  $\Lambda_{ip}(\mathcal {A})$ is a so-called exclusion set for the Brauer-type set $  \mathcal {K}(\mathcal {A})$.
Consider again the tensor $A$ in Remak \ref{rmk2.1}. The sets $\Theta_{ij} (\mathcal {A})$, $ j\neq i$ are drawn in Figure 3,
and their union $ \Theta(\mathcal {A})$ are drawn in Figure 4. It is easy to see that $\sigma (\mathcal {A} )\subseteq \Theta (\mathcal{A})$ and $ \Theta (\mathcal
{A}) \subseteq \mathcal {K}(\mathcal{A})$. In addition, by the relationship of $  \mathcal {K}(\mathcal{A})$ and $\Gamma(\mathcal {A})$,
that is, $\mathcal {K}(\mathcal{A})  \subseteq \Gamma(\mathcal {A})$, we have
\[  \Theta (\mathcal
{A}) \subseteq \mathcal {K}(\mathcal{A}) \subseteq \Gamma(\mathcal {A}).\]
However, $ \Theta (\mathcal {A}) \subseteq  \Omega(\mathcal {A})$ may not hold in general, which can be shown by Figure 2 and Figure 4.



Similarly to Corollary \ref{cor0} and Corollary \ref{cor1}, we can obtain the following results from Theorem  \ref{nth3.1}.

\begin{corollary}  Let $\mathcal {A} =(a_{i_1i_2\cdots i_n})\in \mathbb{C}^{[m,n]}$. If for any  $i,j\in N$ and $j\neq i$,
either
\[ \left(|a_{i\cdots i}|-r_i^j(\mathcal
{A})\right)|a_{j\cdots j}|> |a_{ij\cdots j}|r_j(\mathcal
{A})\]
or
\[\left(|a_{i\cdots i}|+r_i^p(\mathcal {A})\right)|a_{p\cdots p}|< |a_{ip\cdots p}|(2|a_{pi\cdots i}|- r_p(\mathcal {A}))~for ~some p\neq i,\]
then $det (\mathcal {A}) \neq 0 $.
\end{corollary}

\begin{corollary}
 Let $A=(a_{ij})$ be a complex matrix. Then
\[
\sigma(A)\subseteq \Theta (
A) =\bigcup\limits_{i,j\in N,\atop j\neq i} \Theta_{ij} (A),\]
where $ \Theta_{ij} (A)=\mathcal {K}_{ij}(A)\backslash \Lambda_{i}(A) $, $\Lambda_{i}(A)=\bigcup\limits_{p\neq i} \Lambda_{ip}(A)$
and
\[\Lambda_{ip}(A)=\left\{z\in
\mathbb{C}:(|\lambda-a_{ii}|+r_i^p(A))|\lambda-a_{pp}|< |a_{ip}|(2|a_{pi}|- r_p(A)) \right\}.  \]
\end{corollary}

\section{Conclusions}
In this paper, we exclude some proper subsets respectively, which do not include any eigenvalues of a tensor, from the
Ger\v{s}gorin eigenvalue inclusion set $\Gamma(\mathcal {A})$ of tensors in \cite{Qi3} and the Brauer-type eigenvalue inclusion set $\mathcal {K}(\mathcal{A}) $ in \cite{Liet0}
to give two new eigenvalue inclusion sets $\Omega(\mathcal {A})$ and $\Theta (\mathcal {A}) $ with
\[\Omega(\mathcal {A}) \subseteq \Gamma(\mathcal {A}), ~and ~ \Theta (\mathcal {A})\subseteq \mathcal {K}(\mathcal{A}).\]
Besides the sets $\Gamma(\mathcal {A})$ and $\mathcal {K}(\mathcal{A}) $, there are another eigenvalue inclusion sets, such as the sets in \cite{Bu,Bu1,Hu1,Liet1,Liet2,Liet3}.
Hence, for these sets it is interesting to find their proper subsets which do not include any eigenvalue of a tensor to exclude them.

\section*{Acknowledgements}
This work is supported by National Natural Science Foundations of
China (11601473 and 11361074), the National Natural Science Foundation of Zhejiang Province (LY14A010007, LQ14G010002), Ningbo Natural
Science Foundation (2015A610173), and CAS "Light of West China" Program.







\end{document}